\documentclass[a4paper,12pt]{amsart}
\usepackage{graphicx} \usepackage{amsmath}
\usepackage{latexsym,lscape}
\usepackage{amsthm,float}
\usepackage{autograph,epic,latexsym,bezier,caption,amsbsy,psfrag,color,enumerate,amsfonts,amsmath,amscd,amssymb}
\usepackage{epsfig}
\usepackage{rotating}
\usepackage{graphics}
\usepackage[latin1]{inputenc}

\setloopdiam{17}\setprofcurve{10}
\newtheorem{defi}{Definition}[section]
\newtheorem{theorem}[defi]{Theorem}
\newtheorem{lemma}[defi]{Lemma}
\newtheorem{prop}[defi]{Proposition}
\newtheorem{corollary}[defi]{Corollary}

\newtheorem{example}[defi]{Example}

\newtheorem{remark}[defi]{Remark}
\newcommand{\C}{\mathbb{C}}
\newcommand{\VVec}{\textrm{Vec}}

\begin{document}

\thanks{The first author was supported by Austrian Science Fund project FWF P24028-N18.}

\title{Wreath product of matrices}

\author{Daniele D'Angeli}
\address{Daniele D'Angeli, Institut f\"{u}r mathematische Strukturtheorie (Math C)\\
Technische Universit\"{a}t Graz \ \ Steyrergasse 30, 8010 Graz, Austria}
\email{dangeli@math.tugraz.at}
\author{Alfredo Donno}
\address{Alfredo Donno, Universit\`{a} degli Studi Niccol\`{o} Cusano - Via Don Carlo Gnocchi, 3 00166 Roma, Italia \\
Tel.: +39 06 45678356, Fax: +39 06 45678379}
\email{alfredo.donno@unicusano.it}

\maketitle

\begin{abstract}
We introduce a new matrix product, that we call the wreath
product of matrices. The name is inspired by the analogous product for
graphs, and the following important correspondence is proven: the wreath product of the adjacency matrices
of two graphs provides the adjacency matrix of the wreath product of
the graphs. This correspondence is exploited in order to study the spectral
properties of the famous Lamplighter random walk: the spectrum is
explicitly determined for the \lq\lq Walk or switch\rq\rq model on a complete graph of any size,
with two lamp colors. The investigation of the spectrum of the matrix wreath product is actually developed for the more general case where the second factor is a circulant matrix. Finally, an application to the study of generalized Sylvester matrix equations is treated.
\end{abstract}

\begin{center}
\end{center}

\section{Introduction}

A classical tool to deal with combinatorial, probabilistic and analytical problems related to finite graphs is given by the corresponding adjacency matrix, whose rows and columns are indexed by the vertices of the graph, and where the adjacency of two vertices in the graph corresponds to a nonzero entry at the intersection of the corresponding row and column. A natural question arising in this setting asks what would be the appropriate product between matrices associated with a given product of graphs. An easy example is given by the direct product of graphs, whose adjacency matrix can be represented by the Kronecker product of the corresponding adjacency matrices. Similarly, the Cartesian product and the lexicographic product of graphs correspond to the so-called crossed and nested product of matrices, respectively (see \cite{aap}). The reader can refer, for instance, to the papers \cite{imrich, sabidussi, sabidussi2} for definitions and properties of these graph products (see also the beautiful handbook \cite{imrichbook}).\\ In this paper, we define an opportune and general operation between two square matrices $A$ and $B$: the {\it wreath product} $A\wr B$. We show some interesting algebraic properties of such product, focusing our attention on the case in which $B$ is a circulant matrix. In this case, we are able to provide a reduction formula for the spectrum of the matrix $A\wr B$. More precisely, even if the matrix $A\wr B$ has order $nm^n$ when $A$ has order $n$ and $B$ has order $m$, we prove that its spectrum can be explicitly determined by computing the spectrum of much smaller matrices of order $n$.\\
It turns out that the wreath product of matrices is the matrix-analogue of the classical wreath product of graphs (see, for instance, \cite{erschler}): the adjacency matrix of the wreath product of the graphs $\mathcal{G}$ and $\mathcal{G}'$ is given by $A\wr B$, if $A$ and $B$ are the adjacency matrices of the graphs $\mathcal{G}$ and $\mathcal{G}'$, respectively. Sometimes, operations on matrices can have an interesting probabilistic interpretation, since they can be used to model Markov chains. We want to mention here the papers \cite{aap, generalized}, where two families of Markov chains are introduced by using suitable Kronecker products, and the papers \cite{orthogonal, aam}, where these models and their spectral properties are investigated in connection with random walks on trees or on more general combinatorial structures.\\ In our case, the connection with the probability is achieved by the notion of Lamplighter random walk. This is a well known model in the literature, and several papers have been devoted to its analysis, mainly in the case where the underlying graph is the discrete line. Spectral computations for the Lamplighter random walk and related graphs have been developed in the infinite setting \cite{bar, woe}, as well as in the finite setting \cite{sca1, sca2}. In this paper, we prove that the spectral analysis of the matrix $A\wr B$ in the case in which $A$ is the adjacency matrix of a regular graph, provides the spectrum of the Lamplighter random walk on such a graph (see also \cite{ijgt}, where the Lamplighter random walk is studied in connection with the zig-zag product of graphs): an explicit computation is performed for the complete graph. Finally, we apply our results to a uniqueness result for a special type of the so-called Sylvester matrix equations. This kind of equations plays a central role in many areas of applied mathematics, and in particular in systems and control theory. We show that the wreath product allows to represent the coefficient matrix of a system of equations, whose solution coincides exactly with the solution of the generalized Sylvester matrix equation.

The paper is organized as follows. In Section \ref{section2}, we introduce the key definition of the wreath product of two matrices, and we list its basic algebraic properties. In particular, the Lemma \ref{lemmadani0} is devoted to the description of the block structure of the product matrix. In Section \ref{section3}, the correspondence with the wreath product of graphs is proven (Theorem \ref{thmcorrespondence}), and in Corollary \ref{corollarylamp} the transition matrix of the Lamplighter random walk is described, in terms of the wreath product of matrices. In Section \ref{section4}, a linear map $F$ is defined, in order to study the centralizer of a given matrix with respect to the wreath product (Corollary \ref{corollarycommutativity}). The Section \ref{sectionspectrum} is entirely devoted to the investigation of the spectral properties of $A\wr B$, with a circulant $B$, via a reduction argument allowing to describe the spectrum of the matrix product by computing the spectrum of a family of matrices having the same order as $A$ (Theorem \ref{theoremspectrummain}); the spectrum is explicitly described in the case of $A$ of order $2$ (Corollary \ref{caseA2}), and in the case of a diagonal $A$ and a uniform $B$ (Theorem \ref{theoremdeterminant}). An application of these computations provides the explicit spectrum of the Lamplighter random walk on the complete graph (Theorem \ref{spectrumlamp}). Finally, in Section \ref{sylvestersection}, we show that the so-called generalized Sylvester matrix equation can be rewritten by means of the wreath product of matrices (Proposition \ref{prop_sistema}), and we give in Corollary \ref{uniqueness} a sufficient condition for the uniqueness of the solution to this problem.

\section{Wreath product of matrices: definition and basic properties}\label{section2}

We introduce in this section the main definition of the paper, that is, the wreath product of two square matrices of any order. We also list a series of properties of this operation, that are a straightforward consequence of the given definition. In Subsection \ref{subsection21}, we study in the detail the block structure of the matrix $A\wr B$, and we deduce some interesting properties of its determinant.

We denote by $\mathcal{M}_{m\times n}(\mathbb{C})$ the complex vector space of matrices with $m$ rows and $n$ columns over the complex field. We write $\mathcal{M}_n(\mathbb{C})$ instead of $\mathcal{M}_{n\times n}(\mathbb{C})$, and we denote by $I_n$ the identity matrix of order $n$. We recall that the \textit{Kronecker product} of two matrices $A=(a_{ij})_{i=1,\ldots, m; j=1,\ldots, n}\in \mathcal{M}_{m\times n}(\mathbb{C})$ and $B=(b_{hk})_{h=1,\ldots, p; k=1,\ldots, q}\in \mathcal{M}_{p\times q}(\mathbb{C})$ is defined to be the $mp\times nq$ matrix
$$
A\otimes B = \left(
                \begin{array}{ccc}
                  a_{11}B & \cdots & a_{1n}B \\
                  \vdots & \ddots & \vdots \\
                  a_{m1}B & \cdots & a_{mn}B \\
                \end{array}
              \right).
$$
We denote by $A^{\otimes^n}$ the iterated Kronecker product $\underbrace{A\otimes \cdots \otimes A}_{n \textrm{ times}}$. We put $A^{\otimes^0}=1$.
\begin{defi}\label{maindefinition}
Let $A\in \mathcal{M}_n(\mathbb{C})$ and $B\in \mathcal{M}_m(\mathbb{C})$, where $m$ and $n$ are two positive integers. For each $i=1,\ldots, n$, let $C_i = (c_{hk})_{h,k=1,\ldots,n}\in \mathcal{M}_n(\mathbb{C})$ be the matrix defined by
$$
c_{hk}= \left\{
          \begin{array}{ll}
            1 & \hbox{if } h=k=i \\
            0 & \hbox{otherwise.}
          \end{array}
        \right.
$$
The wreath product of $A$ and $B$ is the square matrix of order $n m^n$ defined as
$$
A\wr B = I_m^{\otimes^n} \otimes A + \sum_{i=1}^n I_m^{\otimes^{i-1}}\otimes B \otimes I_m^{\otimes^{n-i}}\otimes C_i.
$$
\end{defi}
We will show (see Corollary \ref{corollarycommutativity}) that the wreath product defined above is in general not commutative, even if the matrices $A$ and $B$ have the same order. On the other hand, the following proposition holds.
\begin{prop}
The wreath product $A\wr B$ satisfies the following quasi-associative laws, with respect to the product by a scalar, and quasi-distributive laws:
\begin{itemize}
\item $(hA)\wr B = h\left(A\wr \left(\frac{1}{h}B\right)\right)$\ \ and\ \ $A\wr (hB) = h\left(\left(\frac{1}{h}A\right)\wr B\right), \quad \forall h\in \mathbb{C}, h\neq 0$.
\item  $A\wr(B+C) = (\lambda A)\wr B + (\mu A)\wr C$\ \ and\ \ $(A+B)\wr C = A\wr (\lambda  C) + B\wr (\mu C)$,\\
where $\lambda,\mu$ are any two complex numbers such that $\lambda+\mu =1$.
\end{itemize}\end{prop}

\begin{proof}
It follows from Definition \ref{maindefinition}.
\end{proof}

 If we denote by $O_p$ the square zero matrix of order $p$, for every $p\in \mathbb{N}$, then we have:
\begin{itemize}
\item $A\wr O_m =  I_m^{\otimes^n} \otimes A$, for every $A\in \mathcal{M}_n(\mathbb{C})$;
\item $O_n \wr B = \sum_{i=1}^n I_m^{\otimes^{i-1}}\otimes B \otimes I_m^{\otimes^{n-i}}\otimes C_i$, for every $B\in \mathcal{M}_m(\mathbb{C})$.
\end{itemize}

Moreover, if one between $A$ or $B$ is the identity matrix, or a multiple, we obtain:
\begin{itemize}
\item $A\wr (hI_m) =  I_m^{\otimes^n} \otimes (A+hI_n)$, for every $A\in \mathcal{M}_n(\mathbb{C})$;
\item $(hI_n) \wr B = hI_{nm^n}+\sum_{i=1}^n I_m^{\otimes^{i-1}}\otimes B\otimes I_m^{\otimes^{n-i}}\otimes C_i$, for every $B\in \mathcal{M}_m(\mathbb{C})$.
\end{itemize}
As a consequence, for all $h,k\in \mathbb{C}$, one has:
$$
(hI_n) \wr (k I_m) = (h+k)I_{nm^n}.
$$
In particular, it follows that
\begin{eqnarray}\label{grossatragedia}
(hI_n) \wr (-h I_m) = O_{nm^n}, \qquad \forall h\in \mathbb{C}:
\end{eqnarray}
in other words, the wreath product of two scalar matrices (whose order is not necessarily the same) with opposite entries is zero.

Finally, in the particular cases where $n=1$ or $m=1$, we have:
$$
(a)\wr B = B+aI_m, \quad \forall a\in \mathbb{C}, B\in \mathcal{M}_m(\mathbb{C}); \quad \qquad A\wr (b) = A+bI_n, \quad \forall b\in \mathbb{C}, A\in \mathcal{M}_n(\mathbb{C}).
$$
\begin{prop}
The wreath product of two symmetric (resp. skew-symmetric, Hermi-\\tian) matrices is a symmetric (resp. skew-symmetric, Hermitian) matrix.
\end{prop}

\begin{proof}
It is known that the Kronecker product is associative, that is: $A \otimes (B\otimes C) = (A\otimes B)\otimes C$. Moreover, the Kronecker product satisfies the following property with respect to the transposition and conjugate transposition:
$$
(A\otimes B)^T = A^T\otimes B^T \qquad (A\otimes B)^\ast = A^\ast\otimes B^\ast.
$$
Therefore, it follows from Definition \ref{maindefinition}, from the linearity of the (conjugate) transposition and from the fact that
$$
I_m=I_m^\ast=I_m^T
$$
that one has:
$$
(A\wr B)^T = (A^T)\wr (B^T) \quad \textrm{and } \ (A\wr B)^\ast = (A^\ast)\wr (B^\ast), \qquad \forall A\in \mathcal{M}_n(\mathbb{C}), B\in \mathcal{M}_m(\mathbb{C}).
$$
The claim follows.
\end{proof}
Finally, we mention the following property of the trace of the wreath product of two matrices:
$$
tr(A\wr B) = m^{n-1}(m tr(A)+n tr(B)), \qquad \forall A\in \mathcal{M}_n(\mathbb{C}), B\in \mathcal{M}_m(\mathbb{C}),
$$
which is an easy consequence of the well known property $tr(A\otimes B)=tr(A)tr(B)$ of the Kronecker product.

\subsection{The block structure of $A\wr B$}\label{subsection21}
It follows from Definition \ref{maindefinition}, which makes use of the iterated Kronecker product of $n$ matrices of order $m$ and one matrix of order $n$, that the matrix $A\wr B$ has an \textit{iterated block structure}. More precisely, it can be regarded as an $m\times m$ block matrix, whose blocks are matrices of order $nm^{n-1}$, obtained by computing the first Kronecker product in each summand appearing in the definition.
On the other hand, each of these blocks has an $m\times m$ block structure with blocks of order $nm^{n-2}$. This argument can be recursively applied, until we get minimal blocks of order $n$, which may have the form $b_{ij}C_k$, where $b_{ij}$ is an entry of $B$, and $k\in \{1,\ldots, n\}$, or the form $A+\sum_{h=1}^n\mu_hC_h$, with some suitable coefficients $\mu_h\in \mathbb{C}$ (see Example \ref{example2222} for the case $n=m=2$). In the next lemma, we give a detailed description of such blocks.
\begin{lemma}\label{lemmadani0}
Let $A\in \mathcal{M}_n(\mathbb{C})$  and $B=(b_{ij})_{i,j=1,\ldots, m} \in \mathcal{M}_m(\mathbb{C})$.
\begin{enumerate}
\item Let $i,j\in\{1,\ldots, m\}$ and $k\in\{1,\ldots,n\}$, with $i\neq j$. Then the number of blocks of type $b_{ij}C_k$ appearing in $A\wr B$ is $m^{n-1}$. More precisely, in the matrix $A\wr B$ regarded as an $m^n\times m^n$ matrix whose entries are blocks of order $n$:
\begin{itemize}
\item the blocks of type $b_{ij}C_k$, with $k=2,\ldots, n-1$, appear at position
$$
((i-1)m^{n-k}+s+tm^{n-k+1}, (j-1)m^{n-k}+s+tm^{n-k+1}),
$$
with $s=1,\ldots, m^{n-k}$, $t=0,\ldots, m^{k-1}-1$;
\item the blocks of type $b_{ij}C_1$ appear at position
$$
((i-1)m^{n-1}+s, (j-1)m^{n-1}+s),
$$
with $s=1,\ldots, m^{n-1}$;
\item the blocks of type $b_{ij}C_n$ appear at position
$$
(i+tm, j+tm )
$$
with $t=0,\ldots, m^{n-1}-1$.
\end{itemize}

\item Let us order the set $\{b_{11}, \ldots, b_{mm}\}$ so that $b_{ii}<b_{jj}$ if $i<j$. Let
$$
\Omega = \{(b_1,\ldots, b_n) : b_i\in    \{b_{11}, \ldots, b_{mm}\}\}
$$
be the set of ordered $n$-tuples of elements from $\{b_{11}, \ldots, b_{mm}\}$, endowed with the following lexicographic order:
$$
(b_1, \ldots, b_n)< (b_1', \ldots, b_n'), \textrm{ if } i=\min\{h: \ b_h\neq b_h'\} \textrm{ is such that } b_i<b_i'.
$$
 Then the $p$-th diagonal entry of $A\wr B$, regarded as an $m^n\times m^n$ matrix whose entries are blocks of order $n$, is given by
$$
A+\sum_{h=1}^n b_{f_{h}+1,f_{h}+1} C_h,
$$
where $(f_1, \ldots,f _n)$ is such that $p=1+\sum_{h=1}^n f_hm^{n-h}$, with $f_h\in \{0,\ldots, m-1\}$. Moreover, $(f_1, \ldots,f _n)$ is the $p$-th $n$-tuple of $\Omega$, with respect to the lexicographic order introduced above. In particular, there is a bijection between the elements in $\Omega$ and the $m^n$ diagonal blocks of $A\wr B$ .
\end{enumerate}
\end{lemma}

\begin{proof}
(1) Notice that, if we determine the explicit positions of the blocks of type $b_{ij}C_k$, then we automatically get the number of such blocks, which turns out to be equal to $m^{n-1}$. So it is enough to prove the second part of the statement. We treat here only the case $1<k<n$, the extremal cases being similar. First of all, observe that the blocks of type $b_{ij}C_k$ are given by the $k$-th summand
$$
I_m^{\otimes^{k-1}}\otimes B\otimes I_m^{\otimes^{n-k}}\otimes C_k.
$$
More precisely, the factor $b_{ij}\otimes I_m^{\otimes^{n-k}}\otimes C_k$ gives rise to $m^{n-k}$ consecutive blocks $b_{ij}C_k$ occupying the main diagonal of  $I_m^{\otimes^{n-k}}$. This contribution is taken into account by the index $s$ which represents a shift along the main diagonal of the corresponding block. Inside the matrix $A\wr B$, seen as a matrix of order $m^n$ with entries given by $n\times n$ blocks, we have $m^{k-1}$ copies of the same block $b_{ij}\otimes I_m^{\otimes^{n-k}}\otimes C_k$, because of the initial factor $I_m^{\otimes^{k-1}}$. These copies are shifted on the main diagonal by a factor $m^{n-k+1}$. This is taken into account by the index $t$ in the formula. The result follows by observing that the first pair of indices corresponding to a position occupied by $b_{ij}C_k$ is $$
((i-1)m^{n-k}+1, (j-1)m^{n-k}+1).
$$
\noindent (2) Let us think again of the matrix $A\wr B$ as a matrix of order $m^n$, whose entries are square matrices of order $n$. This $m^n\times m^n$ matrix can be regarded as a block matrix of order $m$, whose blocks are matrices of order $m^{n-1}$. Scrolling down, the block $M_k$ of order $m^{n-1}$ occupying the $k$-th diagonal entry of $A\wr B$ is of type
$$
I_m^{\otimes^{n-1}} \otimes A + b_{kk}I_m^{\otimes^{n-1}}\otimes C_1+ \sum_{i=1}^{n-1}I_m^{\otimes^{i-1}} \otimes B \otimes I_m^{\otimes^{n-i-1}}\otimes C_{i+1},
$$
with $k=1,\ldots, m$. In particular, all the blocks of order $n$ in the main diagonal of the $k$-th block $M_k$ of order $m^{n-1}$ of $A\wr B$ contain the summand $A+b_{kk}C_1$. Now notice that the entry $(p,p)$ of $A\wr B$, seen as a block matrix of order $m^n$, belongs to $M_k$ if in the expression $p=1+\sum_{h=1}^n f_hm^{n-h} $ the coefficient $f_1+1$ equals $k$. This implies that at the entry $(p,p)$ we will find a matrix containing the summand $A+b_{f_1+1,f_1+1}C_1$. Now zooming into this $(f_1+1)$-st block $M_{f_1+1}$ of order $m^{n-1}$ (containing $(p,p)$), we regard $M_{f_1+1}$ as a block matrix with blocks (on the main diagonal) $M_1',\ldots, M_m'$ of order $m^{n-2}$. Scrolling down, the blocks occupying the main diagonal  have the form
$$
I_m^{\otimes^{n-2}} \otimes A + b_{kk}I_m^{\otimes^{n-2}}\otimes C_1+ b_{tt}I_m^{\otimes^{n-2}}\otimes C_2  +
$$
$$
+\sum_{i=1}^{n-2} I_m^{\otimes^{i-1}}\otimes B \otimes I_m^{\otimes^{n-i-2}}\otimes C_{i+2},
$$
 with $t=1,\ldots, m$. In particular, all the blocks of order $n$ in the main diagonal of the $t$-th block $M_t'$ of order $m^{n-2}$ in $M_{f_1+1}$, contain the summand $A+b_{kk}C_1+b_{tt}C_2$.  Now notice that the entry $(p,p)$ of $A\wr B$ belongs to the block $M_t'$ (inside $M_k=M_{f_1+1}$) if the coefficient $f_2+1$ equals $t$. This implies that at the entry $(p,p)$ we will find a matrix containing the summand $A+b_{f_1+1,f_1+1}C_1+b_{f_2+1,f_2+1}C_2$. We can apply recursively this argument. It is easy to see that the order in which two different matrices of type $C:=A+\sum_{h=1}^n c_h C_h$ and $C':=A+\sum_{h=1}^n c_h' C_h$ appear in different diagonal entries coincides with the lexicographic order of $\Omega$ defined above. In fact, two different diagonal entries are occupied by different expressions $C$ and $C'$ because at some step of the recursive method they will belong to different submatrices $M_i^{(s)}$ and $M_j^{(s)}$ (with $i<j$, for example) of order $m^{n-s-1}$. This is equivalent to say that the following inequality in the lexicographic order in $\Omega$ holds: $(c_1, \ldots, c_n)<(c_1', \ldots, c_n')$.
 The $+1$ that appears in the equality $p=1+\sum_{h=1}^n f_hm^{n-h}$, takes into account the fact that the position (1,1) corresponds to $f_h=0$ for any $h$. It is clear from the construction that there is a bijection between the set of the $n$-tuples of $f_h$'s over $\{0,1,\ldots, m-1\}$ and the diagonal blocks of $A\wr B$ (regarded as a matrix of order $m^n$). The first of such sets uniquely determines the $n$-tuple $(b_1, \ldots, b_n)\in \Omega$ consisting of $n$ elements from $\{b_{11}, \ldots, b_{mm}\}$. This concludes the proof.
\end{proof}

\begin{example}\label{example2222}\rm
Let $A\in \mathcal{M}_2(\mathbb{C})$ and let $B= \left(
                                                   \begin{array}{ccc}
                                                     b_{11} & b_{12}  \\
                                                     b_{21} & b_{22} \\
                                                   \end{array}
                                                 \right)\in \mathcal{M}_2(\mathbb{C})$. Then $A\wr B$ is a square matrix of order $8$, which can be regarded as a square matrix of  order $4$, whose entries are square submatrices of order $2$:
$$
 \left(\begin{array}{cc|cc}
    A+b_{11}(C_1+C_2) & b_{12}C_2 & b_{12}C_1 & 0 \\
    b_{21}C_2 & A+b_{11}C_1+b_{22}C_2 & 0 & b_{12}C_1 \\
\hline
    b_{21}C_1 & 0 & A+b_{22}C_1+b_{11}C_2 & b_{12}C_2 \\
    0 & b_{21}C_1 & b_{21}C_2 & A+b_{22}(C_1+C_2) \\
\end{array}
\right).
$$
Observe that each block of type $b_{ij}C_k$, with $i\neq j$, appears twice off the main diagonal; moreover, all the blocks on the main diagonal are distinct, and ordered according with the lexicographic order of the pairs $\{(b_{ii},b_{jj}): i,j=1,2\}$.
\end{example}

A result of J. R. Silvester \cite{det} shows that, given a block matrix $M = (M_{ij})_{i,j=1,\ldots, h}\in \mathcal{M}_{nh}(\C)$, where each $M_{ij}$ is a square matrix of order $n$ over a field or a commutative ring, if all the blocks commute with each other, then the determinant $\det M$ of the full matrix equals the determinant of the $n\times n$ matrix $DET(M,n)$, which is the formal determinant of $M$ in terms of its block entries.

\begin{example}\rm
Let $M= \left(
          \begin{array}{ccc}
            M_{11} & M_{12} & M_{13} \\
            M_{21} & M_{22} & M_{23} \\
            M_{31} & M_{32} & M_{33} \\
          \end{array}
        \right)
$, where $M_{ij}$ is a submatrix of order $n$ of $M$, for all $i,j$, and suppose that $M_{ij}M_{rs}=M_{rs}M_{ij}$, for all $i,j,r,s=1,2,3$. Then $\det M$ equals the determinant of the $n\times n$ matrix
$$
M_{11}M_{22}M_{33} + M_{12}M_{23}M_{31}+M_{13}M_{21}M_{32}-M_{12}M_{21}M_{33}-M_{13}M_{22}M_{31}-M_{11}M_{23}M_{32}.
$$
\end{example}

Keeping in mind this property, we want to give some conditions to get the commutativity of the blocks of the matrix $A\wr B$ described in Lemma \ref{lemmadani0}.

\begin{lemma}
Let $A\in \mathcal{M}_n(\mathbb{C})$  and $B\in \mathcal{M}_m(\mathbb{C})$ be two nonzero matrices. Then the blocks of the matrix $A\wr B$, regarded as a matrix of order $m^{n}$ whose entries are matrices of order $n$, commute with each other if and only if $A$ is diagonal.
\end{lemma}
 \begin{proof}
We have seen in Lemma \ref{lemmadani0} that the blocks of order $n$ of the matrix
$$
A\wr B = I_m^{\otimes^n} \otimes A + \sum_{i=1}^n I_m^{\otimes^{i-1}}\otimes B \otimes I_m^{\otimes^{n-i}}\otimes C_i
$$
are either of type $A+\sum_{h=1}^nb_{f_h+1,f_h+1}C_h$ (on the main diagonal) or of type $b_{ij}C_k, i\neq j$ (off the main diagonal). It is clear that, if at least one between $A$ or $B$ is the zero matrix, then all these blocks commute with each other. For this reason, in this lemma we restrict our attention to the case $A\neq O_n$ and $B\neq O_m$. First of all, note that the matrices $C_i$'s commute with each other, more precisely we have
$$
C_iC_j = \left\{
           \begin{array}{ll}
             C_i & \hbox{if } i=j \\
             O_n & \hbox{if } i\neq j.
           \end{array}
         \right.
$$
The matrix $A$ commutes with all the $C_i$'s if and only if it is a diagonal matrix. In fact, $A C_i$ is the matrix with all zeros, except for the $i$-th column which equals that of $A$. On the other hand, $C_iA$  is the matrix with all zeros, except for the $i$-th row which equals that of $A$. Therefore, fixed an index $i\in \{1,\ldots, n\}$, we have $AC_i=C_iA$ if and only if one has $a_{ij}=a_{ji}=0$, for each $j\neq i$. Now it suffices to apply the distributivity of the matrix multiplication and the claim follows.
\end{proof}

\begin{remark}\rm
We observe that, in this setting of a diagonal matrix $A$, after computing the determinant of the block matrix $DET(A\wr B, n)$, we get a polynomial expression in terms of the matrices $A$ and $C_i$'s; in fact, it follows from Lemma \ref{lemmadani0} that the diagonal entries are of type $A+\sum_{h=1}^n b_{f_h+1,f_h+1} C_h$, with $f_h\in \{0,\ldots, m-1\}$, and the entries off the main diagonal are of type $b_{ij}C_h$, for some $i,j\in \{1,\ldots, m\}$, $i\neq j$, and $ h \in \{1,\ldots,n\}$.
Moreover, the matrix corresponding to the expression given by this polynomial (whose determinant equals the determinant of the full matrix $A\wr B$) is diagonal as a sum of products of diagonal matrices. This implies that the matrix $DET(A\wr B,n)$ has full rank (and so $A\wr B$ has full rank) if and only if the matrix $DET(A\wr B,n)$ contains no zeros on the main diagonal.
\end{remark}

We can summarize our discussion with the following result.

\begin{prop}
Let $A\in \mathcal{M}_n(\C)$ be a diagonal matrix. Then the only nonzero contribution (up to the sign) to $DET(A\wr B, n)$ are of the following types:
\begin{enumerate}
\item $\left(\prod_{h=1}^{m^n} t_h\right) C_k$, for some $t_h\in \{b_{ij}: i\neq j\}$ and $k\in\{1,\ldots, n\}$;\\
\item $\prod_{r=1}^{s}\left(A+\sum_{q_r=1}^n b_{q_r} C_{q_r}\right)\times \left(\prod_{h=1}^{m^n-s} t_h\right) C_k$, for some $b_{q_r}\in \{b_{ii}: i=1,\ldots, n\}$, $1\leq s\leq m^n$, $t_h\in \{b_{ij}: i\neq j\}$ and $k\in\{1,\ldots, n\}$.
\end{enumerate}
\end{prop}
\begin{proof}
The proof easily follows from Lemma \ref{lemmadani0}. The only nonzero summands that appear in the determinant are either those in which each factor is a scalar multiple of the same matrix $C_k$, for some $k=1,\ldots, n$, or those in which the factors are scalar multiples of the same matrix $C_k$, together with some blocks from the main diagonal. These two cases are described by $(1)$ and $(2)$ in the assertion, respectively.
\end{proof}

In Corollary \ref{otello}, we will give an explicit description of the determinant of $A\wr B$ for a diagonal $A$ and a uniform $B$, by using the spectral analysis developed in Section \ref{sectionspectrum} for $A\wr B$, in the more general case of a circulant matrix $B$.


\section{Wreath product of graphs and the Lamplighter random walk}\label{section3}

We start this section by recalling the classical definition of wreath product of two graphs \cite{gc,erschler}, then we show that it is a graph-analogue of the wreath product of matrices introduced in Definition \ref{maindefinition}. More precisely, we prove in Theorem \ref{thmcorrespondence} that the wreath product of the normalized adjacency matrices of two (regular) finite graphs provides the normalized adjacency matrix of the wreath product of the graphs. We use the symbol $\sim$ to denote adjacency of two vertices in a graph.

 \begin{defi}\label{defierschler}
Let $\mathcal{G}_1=(V_1, E_1)$ and $\mathcal{G}_2=(V_2,E_2)$ be
two finite graphs. The wreath product $\mathcal{G}_1\wr
\mathcal{G}_2$ is the graph with vertex set $V_2^{V_1}\times V_1=
\{(f,v) | f:V_1\to V_2, \ v\in V_1\}$, where $(f,v)\sim (f',v')$ if:
\begin{enumerate}
\item either $v=v'=:\overline{v}$ and $f(w)=f'(w)$ for every $w\neq \overline{v}$,
and $f(\overline{v})\sim f'(\overline{v})$ in $\mathcal{G}_2$;
\item or $f(w)=f'(w)$, for every $w\in
V_1$, and $v\sim v'$ in $\mathcal{G}_1$.
\end{enumerate}
\end{defi}
Observe that, if $\mathcal{G}_1$ is a
$d_1$-regular graph on $n_1$ vertices and
$\mathcal{G}_2$ is a $d_2$-regular graph on $n_2$ vertices, then $\mathcal{G}_1\wr \mathcal{G}_2$ is a
$(d_1+d_2)$-regular graph on $n_1n_2^{n_1}$ vertices.

The wreath product of graphs represents a graph-analogue of the
classical wreath product of groups: more precisely (see Theorem 2.1 in \cite{gc}), it is true that, with a particular choice of the
generating sets, the wreath
product of the Cayley graphs of two finite groups is the Cayley
graph of the wreath product of the groups (see also a generalization of this construction therein).

It is a classical fact (see, for instance, \cite{woe}) that the simple random walk on the wreath product $\mathcal{G}_1\wr
\mathcal{G}_2$ of two graphs is the so-called \textit{Lamplighter random walk}, according with the following interpretation:
suppose that at each vertex of $\mathcal{G}_1$ there is a lamp, whose possible states (or colors) are
represented by the vertices of $\mathcal{G}_2$ (the \textit{color graph}), so that the vertex $(f,v)$ of $\mathcal{G}_1\wr
\mathcal{G}_2$ represents the configuration of the $|V_1|$ lamps at each vertex of $\mathcal{G}_1$ (for each vertex $u\in V_1$, the lamp at $u$ is in the state $f(u) \in V_2$), together with the position $v$ of a lamplighter walking on the graph $\mathcal{G}_1$. At each step, the lamplighter may either go to a neighbor of the
current vertex $v$ and leave all lamps unchanged (this situation corresponds to edges of type (2) in $\mathcal{G}_1\wr
\mathcal{G}_2$), or he may stay at the vertex $v \in \mathcal{G}_1$, but he changes the state of the lamp which is in $v$ to a
neighbor state in $\mathcal{G}_2$ (this situation corresponds to edges of type (1) in $\mathcal{G}_1\wr
\mathcal{G}_2$): this model is also called the
\lq\lq \textit{Walk or switch}\rq\rq model.

Recall that the \textit{normalized adjacency matrix} of an (undirected) $d$-regular graph $\mathcal{G}=(V,E)$ is the square matrix $A=(a_{ij})$ of order $|V|$, whose entry $a_{ij}$ equals the number of (undirected) edges connecting the $i$-th and the $j$-th vertex of $\mathcal{G}$, divided by the degree $d$. We are now ready to prove the following theorem.

\begin{theorem}\label{thmcorrespondence}
Let $A_1$ be the normalized adjacency matrix of a $d_1$-regular graph $\mathcal{G}_1=(V_1, E_1)$ and let $A_2$ be the normalized adjacency matrix of a $d_2$-regular graph $\mathcal{G}_2=(V_2, E_2)$. Put $|V_i|=n_i$, for $i=1,2$. Then the wreath product $\left(\frac{d_1}{d_1+d_2}A_1\right)\wr \left(\frac{d_2}{d_1+d_2}A_2\right)$ is the normalized adjacency matrix of the graph wreath product $\mathcal{G}_1\wr \mathcal{G}_2$.
\end{theorem}
\begin{proof}
First of all, observe that the wreath product $\left(\frac{d_1}{d_1+d_2}A_1\right)\wr \left(\frac{d_2}{d_1+d_2}A_2\right)$ is a square matrix of order $n_1n_2^{n_1}$, which can be rewritten as
$$
\frac{d_1}{d_1+d_2}\left(I_{n_2}^{\otimes^{n_1}} \otimes A_1\right) + \frac{d_2}{d_1+d_2}\left(\sum_{i=1}^{n_1}I_{n_2}^{\otimes^{i-1}}\otimes A_2\otimes I_{n_2}^{\otimes^{n_1-i}}\otimes C_i\right).
$$
In order to show the assertion, it suffices to order the vertices $\{(f,v): f: V_1\to V_2, v\in V_1\}$ of $\mathcal{G}_1\wr \mathcal{G}_2$ in the following suitable way.\\ \indent First of all, let $V_1=\{v_1, \ldots, v_{n_1}\}$ and $V_2=\{u_1, \ldots, u_{n_2}\}$, both with the natural ordering such that $v_i$ precedes $v_{i+1}$ in $V_1$, and $u_j$ precedes $u_{j+1}$ in $V_2$. Observe that these orderings reflect on the rows and columns of the matrices $A_1$ and $A_2$, respectively. Let us introduce a lexicographic ordering in $V_2^{V_1}$, by saying that $f$ precedes $g$ in $V_2^{V_1}$ if there exists an index $i\in \{1,\ldots, n_1\}$ such that $f(v_i)$ precedes $g(v_i)$ in $V_2$, and $f(v_h) = g(v_h)$ for each $1\leq h <i$. Finally, we say that $(f,v_i)$ precedes $(g,v_j)$ in $V_2^{V_1}\times V_1$ if
\begin{itemize}
\item either $f$ precedes $g$ in $V_2^{V_1}$;
\item or $f=g$, and $v_i$ precedes $v_j$ in $V_1$.
\end{itemize}
Observe that, with this ordering, the summand
$$
I_{n_2}^{\otimes^{n_1}} \otimes A_1
$$
corresponds to edges of type (2) in $\mathcal{G}_1\wr \mathcal{G}_2$, namely $(f,v_h)\sim (f,v_k)$, for some $f:V_1\to V_2$ and with $v_h\sim v_k$ in $\mathcal{G}_1$. As the total degree of the graph $\mathcal{G}_1\wr \mathcal{G}_2$ is $d_1+d_2$, and the degree of $\mathcal{G}_1$ is $d_1$, such a summand must be multiplied by the factor $\frac{d_1}{d_1+d_2}$.

On the other hand, the summand
$$
I_{n_2}^{\otimes^{i-1}}\otimes A_2\otimes I_{n_2}^{\otimes^{n_1-i}}\otimes C_i
$$
corresponds to edges of type (1) in $\mathcal{G}_1\wr \mathcal{G}_2$: more precisely, $(f,v_i)\sim (g,v_i)$, with $f(v_j) = g(v_j)$ for each $v_j\neq v_i$ and $f(v_i)\sim g(v_i)$ in $\mathcal{G}_2$. In particular, the matrix $C_i$ takes into account the fact that the lamplighter stays at the vertex $v_i$. Since the total degree of $\mathcal{G}_1\wr \mathcal{G}_2$ is $d_1+d_2$, and the degree of $\mathcal{G}_2$ is $d_2$, this second summand must be multiplied by the factor $\frac{d_2}{d_1+d_2}$. This concludes the proof.
\end{proof}

\begin{corollary}\label{corollarylamp}
Let $\mathcal{G}_1=(V_1, E_1)$ be a regular graph of degree $d_1$, and let $\mathcal{G}_2=(V_2, E_2)$ be a regular graph of degree $d_2$. Let $A_1$ (resp. $A_2$)  be the normalized adjacency matrix of $\mathcal{G}_1$ (resp. $\mathcal{G}_2$). Then the wreath product $\left(\frac{d_1}{d_1+d_2}A_1\right)\wr \left(\frac{d_2}{d_1+d_2}A_2\right)$ is the transition matrix of the \lq\lq Walk or switch\rq\rq Lamplighter random walk on the graph $\mathcal{G}_1$, with color graph $\mathcal{G}_2$.
\end{corollary}

\begin{example}\rm
Let $\mathcal{G}_1$ be the cyclic graph of length $3$, and consider on the graph $\mathcal{G}_1$ the \lq\lq Walk or switch\rq\rq Lamplighter random walk, where the lamp set consists of exactly two colors, let they be $0$ and $1$: such colors can be identified with the vertices of the segment graph $\mathcal{G}_2$ (see Figure \ref{figuragrafi}). In other words, we can write $V_1=\{0,1,2\}$ and $V_2=\{0,1\}$. The normalized adjacency matrices are
$$
A_1= \frac{1}{2}\left(
                  \begin{array}{ccc}
                    0 & 1 & 1 \\
                    1 & 0 & 1 \\
                    1 & 1 & 0 \\
                  \end{array}
                \right)      \qquad \textrm{and } \qquad A_2= \left(
                                                               \begin{array}{cc}
                                                                 0 & 1 \\
                                                                 1 & 0 \\
                                                               \end{array}
                                                             \right).
$$
Then the wreath product $\left(\frac{2}{3}A_1\right)\wr \left(\frac{1}{3}A_2\right)$ is the transition matrix of the \lq\lq Walk or switch\rq\rq  \\Lamplighter random walk on $\mathcal{G}_1$. The graph $\mathcal{G}_1\wr \mathcal{G}_2$ is represented in Figure \ref{figuraprodotto}. The vertex denoted by $u_1u_2u_3,v_i$ represents the situation where the lamplighter is at the vertex $v_i$ of $\mathcal{G}_1$, and $u_j\in V_2$ is the state of the lamp at the vertex $v_j\in V_1$.
\begin{center}
\begin{figure}[h]
\begin{picture}(400,52)\unitlength=0,3mm
\letvertex A=(130,5)\letvertex B=(180,5)\letvertex C=(155,48)   \letvertex D=(250,25)\letvertex E=(300,25)

\put(124,-8){$0$}\put(179,-8){$1$}\put(152,54){$2$}\put(246,10){$0$}\put(297,10){$1$}


 \drawvertex(A){$\bullet$}\drawvertex(B){$\bullet$}\drawvertex(C){$\bullet$}\drawvertex(D){$\bullet$}
\drawvertex(E){$\bullet$}

\drawundirectededge(A,B){}\drawundirectededge(A,C){}\drawundirectededge(B,C){}
\drawundirectededge(D,E){}
\end{picture}\caption{The graphs $\mathcal{G}_1$ and $\mathcal{G}_2$.}\label{figuragrafi}
\end{figure}
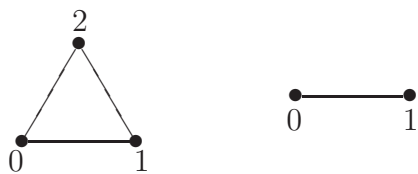
\end{center}
\centering
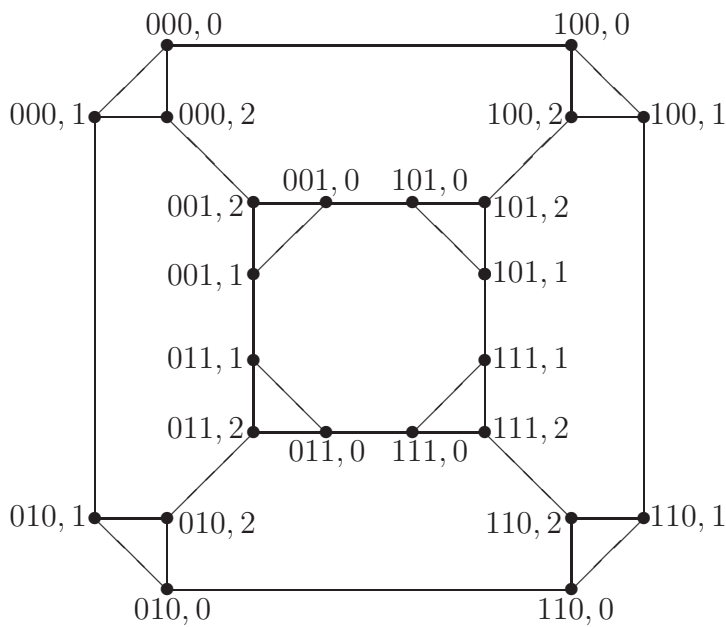
\begin{figure}[h]
\begin{picture}(250,260)\unitlength=0,19mm
\letvertex A=(10,340)\letvertex B=(10,60)\letvertex C=(60,10)\letvertex D=(340,10)\letvertex E=(390,60)\letvertex F=(390,340)\letvertex G=(340,390)
\letvertex H=(60,390)\letvertex I=(60,340)\letvertex L=(60,60)\letvertex M=(340,60)\letvertex N=(340,340)\letvertex O=(120,280)
\letvertex P=(120,230)\letvertex Q=(120,170)\letvertex R=(120,120)\letvertex S=(170,120)\letvertex T=(230,120)\letvertex U=(280,120)
\letvertex V=(280,170)\letvertex Z=(280,230)\letvertex X=(280,280)\letvertex Y=(230,280)\letvertex W=(170,280)

\put(45,398){$000,0$}\put(-49,335){$000,1$}\put(68,335){$000,2$}\put(37,-10){$010,0$}\put(-49,55){$010,1$}\put(68,50){$010,2$}
\put(327,398){$100,0$}\put(394,335){$100,1$}\put(281,335){$100,2$}\put(316,-10){$110,0$}\put(394,55){$110,1$}\put(280,50){$110,2$}
\put(60,223){$001,1$}\put(140,289){$001,0$}\put(60,271){$001,2$}\put(60,116){$011,2$}\put(144,100){$011,0$}\put(60,165){$011,1$}
\put(285,116){$111,2$}\put(215,100){$111,0$}\put(285,165){$111,1$}\put(285,270){$101,2$}\put(215,289){$101,0$}\put(285,225){$101,1$}


\drawvertex(A){$\bullet$}\drawvertex(B){$\bullet$}\drawvertex(C){$\bullet$}\drawvertex(D){$\bullet$}
\drawvertex(E){$\bullet$}\drawvertex(F){$\bullet$}\drawvertex(G){$\bullet$}\drawvertex(I){$\bullet$}\drawvertex(M){$\bullet$}\drawvertex(N){$\bullet$}
\drawvertex(H){$\bullet$}\drawvertex(L){$\bullet$}\drawvertex(O){$\bullet$}\drawvertex(T){$\bullet$}
\drawvertex(P){$\bullet$}\drawvertex(U){$\bullet$}\drawvertex(Q){$\bullet$}\drawvertex(V){$\bullet$}
\drawvertex(R){$\bullet$}\drawvertex(Z){$\bullet$}\drawvertex(Y){$\bullet$}\drawvertex(Z){$\bullet$}
\drawvertex(S){$\bullet$}\drawvertex(X){$\bullet$}\drawvertex(W){$\bullet$}

\drawundirectededge(A,B){}\drawundirectededge(A,H){}\drawundirectededge(A,I){}\drawundirectededge(B,L){}
\drawundirectededge(B,C){}\drawundirectededge(C,L){}\drawundirectededge(C,D){}\drawundirectededge(D,E){}
\drawundirectededge(D,M){}\drawundirectededge(E,M){}\drawundirectededge(E,F){}
\drawundirectededge(F,G){}\drawundirectededge(F,N){}\drawundirectededge(G,N){}\drawundirectededge(G,H){}
\drawundirectededge(H,I){}\drawundirectededge(I,O){}\drawundirectededge(O,P){}\drawundirectededge(O,W){}\drawundirectededge(P,Q){}
\drawundirectededge(Q,R){}\drawundirectededge(Q,S){}\drawundirectededge(R,L){}
\drawundirectededge(R,S){}\drawundirectededge(S,T){}\drawundirectededge(T,U){}\drawundirectededge(T,V){}
\drawundirectededge(U,M){}\drawundirectededge(U,V){}\drawundirectededge(V,Z){}\drawundirectededge(Z,Y){}\drawundirectededge(Z,X){}
\drawundirectededge(N,X){}\drawundirectededge(X,Y){}\drawundirectededge(Y,W){}\drawundirectededge(P,W){}
\end{picture} \caption{The graph $\mathcal{G}_1\wr \mathcal{G}_2$.}\label{figuraprodotto}
\end{figure}
\end{example}

In Subsection \ref{amilcare} we will determine the spectrum of the Lamplighter random walk with two colors, by using the spectral analysis developed for the wreath product of matrices.


\section{The map $F$}\label{section4}

In order to characterize matrices which commute with respect to the wreath product, we introduce and study in this section a map $F:\mathcal{M}_n(\mathbb{C})\times \mathcal{M}_m(\mathbb{C}) \longrightarrow \mathcal{M}_{n  m^n}(\mathbb{C})$ defined as
\begin{eqnarray*}
F(A,B) = A\wr B, \qquad \forall A\in \mathcal{M}_n(\mathbb{C}), B\in \mathcal{M}_m(\mathbb{C}).
\end{eqnarray*}
We start with the following lemma.
\begin{lemma}
The map $F$ is a homomorphism of vector spaces.
\end{lemma}
\begin{proof}
By definition, we have
\begin{eqnarray*}
F((A,B)+(C,D)) &=& F(A+C,B+D)\\
&=& I_m^{\otimes^n} \otimes (A+C)\\
&+&\sum_{i=1}^n I_m^{\otimes^{i-1}}\otimes (B+D) \otimes I_m^{\otimes^{n-i}}\otimes C_i\\
&=&  I_m^{\otimes^n} \otimes A + \sum_{i=1}^n I_m^{\otimes^{i-1}} \otimes B \otimes I_m^{\otimes^{n-i}}\otimes C_i \\
&+& I_m^{\otimes^n} \otimes C + \sum_{i=1}^n I_m^{\otimes^{i-1}} \otimes D \otimes I_m^{\otimes^{n-i}}\otimes C_i\\
&=& A\wr B + C\wr D\\
&=& F(A,B) + F(C,D),
\end{eqnarray*}
for all $A,C \in \mathcal{M}_n(\mathbb{C})$ and  $B,D \in \mathcal{M}_m(\mathbb{C})$. Now let $h\in \mathbb{C}$. By definition:
\begin{eqnarray*}
F(h(A,B)) &=& F(hA,hB)\\
&=& I_m^{\otimes^n} \otimes (hA) + \sum_{i=1}^n I_m^{\otimes^{i-1}}\otimes (hB) \otimes I_m^{\otimes^{n-i}}\otimes C_i\\
&=&h\left(I_m^{\otimes^n} \otimes A\right) + h\sum_{i=1}^n I_m^{\otimes^{i-1}}\otimes B \otimes I_m^{\otimes^{n-i}}\otimes C_i \\
&=& h\left(I_m^{\otimes^n} \otimes A + \sum_{i=1}^n I_m^{\otimes^{i-1}} \otimes B \otimes I_m^{\otimes^{n-i}}\otimes C_i\right) \\
&=& hF(A,B),
\end{eqnarray*}
for all $A\in \mathcal{M}_n(\mathbb{C}), B\in \mathcal{M}_m(\mathbb{C})$. The lemma is proved.
\end{proof}

\begin{prop}\label{kernel}
The kernel $\ker F$ of the homomorphism $F$ is $\{(hI_n, -hI_m), h\in \mathbb{C}\}$.
\end{prop}
\begin{proof}
Recall that we have already remarked in Equation \eqref{grossatragedia} that $\{(hI_n, -hI_m), h\in \C\}\subseteq \ker F$. We start our proof by showing that $(A,B)\not\in \ker F$ whenever at least one between $A$ and $B$ is not diagonal. First, suppose that the matrix $B=(b_{ij})_{i,j=1,\ldots, m}$ is not diagonal, so that there exist two indices $i'\neq j'$ such that $b_{i'j'}\neq 0$. This implies that the matrix $b_{i'j'} I_m^{\otimes^{n-1}}\otimes C_1$ is nonzero. On the other hand, if we regard now the matrix $F(A,B)=A\wr B$ as a block matrix of order $m$ whose entries are matrices of order $nm^{n-1}$, then it is easy to check that the submatrix occupying the position $(i',j')$ is given exactly by $b_{i'j'} I_m^{\otimes^{n-1}}\otimes C_1$, and so it is nonzero. Hence, in this case, $F(A,B)\neq O_{nm^n}$.\\
Let us suppose now that the matrix $A=(a_{ij})_{i,j=1,\ldots, n}$ is not diagonal. Then there exist two indices $i'\neq j'$ such that $a_{i'j'}\neq 0$. In this case, it is useful to regard $A\wr B$ as a block matrix of order $m^{n}$ whose entries are matrices of order $n$. In such a matrix, the block corresponding to the entry $(1,1)$ is occupied by the matrix
$$
A+\sum_i^n b_{11} C_i= A+b_{11}I_n,
$$
according with Lemma \ref{lemmadani0}. As the entry $a_{i'j'}$ of $A$ is nonzero, this matrix is certainly nonzero. \\
Thus, for the remaining part of the proof, we can suppose that both $A$ and $B$ are diagonal matrices. When $F(A,B)= O_{nm^n}$, then it must be $A+b_{11}I_n=O_{n}$. This equality between matrices gives rise to the $n$ equations
$$
a_{11}+b_{11}=0, \ a_{22}+b_{11}=0,  \ldots, \ a_{nn}+b_{11}=0
$$
and so we have:
$$
h:=a_{11}=a_{22}=\cdots =a_{nn}= -b_{11}.
$$
Analogously, for each $j=1,\ldots, m$, the entry $(j,j)$ of the matrix $F(A,B)=A\wr B$ (seen as a matrix of order $m^{n}$ whose entries are matrices of order $n$) is $A+ b_{11} (I_n-C_n)+ b_{jj} C_n$, according with Lemma \ref{lemmadani0}. Since we have supposed $F(A,B)= O_{nm^n}$ then, for every $j=2,\ldots, m$, it must be:
$$
\left(
  \begin{array}{cccc}
    h+b_{11} &  &  &  \\
     & \ddots &  &  \\
     &  & h+b_{11} &  \\
     &  & & h+b_{jj} \\
  \end{array}
\right) = O_n,
$$
what implies $b_{11}=b_{22}= \cdots = b_{mm}=-h$. The claim follows.
\end{proof}

\begin{corollary}\label{corollarycommutativity}
Let $A$ and $B$ be two square matrices of order greater than $1$. Then $A\wr B = B\wr A$ if and only if $A$ and $B$ have the same order and differ by a multiple of the identity matrix.
\end{corollary}

\begin{proof}
First of all observe that, if we exclude the trivial case where at least one between $A$ and $B$ has order $1$, the matrices $A\wr B$ and $B\wr A$ have the same order if and only if the matrices $A$ and $B$ have the same order, let it be $n$. In this case, both $A\wr B$ and $B\wr A$ have order $n^{n+1}$. Now we have
$$
A\wr B = B\wr A \Leftrightarrow F(A,B) = F(B,A) \Leftrightarrow F(A-B,B-A) = O_{n^{n+1}}  \Leftrightarrow A-B = hI_n,
$$
for some $h\in \mathbb{C}$, according with Proposition \ref{kernel}.
\end{proof}
In other words, two matrices $A$ and $B$ commute if and only if they differ by a multiple of the identity matrix. This can also be reformulated by stating that the centralizer of a matrix $A\in \mathcal{M}_n(\mathbb{C})$ is given by $Centr(A) = \{A+hI_n, h\in \mathbb{C}\}$.

\section{Spectral computations}\label{sectionspectrum}

In this section we investigate the spectral properties of the wreath product $A\wr B$ of two matrices. We focus our attention on the special case of a circulant matrix $B$, that is, we assume:
\begin{eqnarray}\label{aggiuntaora}
B = \begin{pmatrix} b_0 & b_1 &  &  & b_{m-1} \\
 b_{m-1} & b_0 & b_1 &  &  \\
 & \ddots & \ddots & \ddots  & \\
       &  & \ddots & \ddots  & b_1 \\
 b_1 &  &  & b_{m-1} & b_0
 \end{pmatrix}, \qquad \textrm{with } b_i\in \mathbb{C}, \forall i=0,\ldots, m-1.
\end{eqnarray}
From now on, we will denote by $Circ_m(\mathbb{C})$ the vector subspace of $\mathcal{M}_m(\C)$ consisting of circulant matrices. The reader can refer to \cite{davis} as an exhaustive monograph on circulant matrices.

Before starting our spectral computations, we discuss some further properties of the map $F$ defined in Section \ref{section4}, when it is restricted to the complex vector space $M_n(\C)\times Circ_m(\C)$. Recall that a basis of the vector space $M_n(\C)$ is given by
$$
\{E_{ij}: i,j =1,\ldots,n\},
$$
where the entry $e_{hk}$ of $E_{ij}$ satisfies
$$
e_{hk}=\left\{
                                                                                                       \begin{array}{ll}
                                                                                                         1, & \hbox{if } (h,k)=(i,j)\\
                                                                                                         0, & \hbox{otherwise}
                                                                                                       \end{array}
                                                                                                     \right.
\quad  \textrm{ for } h,k=1,\ldots, n,
$$
whereas a basis of the vector space $Circ_m(\C)$ is
$$
\{Circ_i: i=0,\ldots,m-1\},
$$
where the entry $c_{hk}$ of $Circ_i$ satisfies the condition
$$
c_{hk}=\left\{
                                                                                                       \begin{array}{ll}
                                                                                                         1, & \hbox{if } k-h\equiv i \mod m \\
                                                                                                         0, & \hbox{otherwise}
                                                                                                       \end{array}
                                                                                                     \right.
\quad \textrm{ for } h,k=0,\ldots, m-1.
$$
Observe that the following identities hold:
$$
Circ_i^T = Circ_{m-i}, \qquad Circ_i Circ_j = Circ_{i+j} = Circ_j Circ_i, \qquad \forall i,j = 0,1,\ldots, m-1.
$$
Note that the sum $i+j$ must be taken modulo $m$. The basis of the vector space  $M_n(\C)\times Circ_m(\C)$ is defined in the obvious way. It follows from the definition of wreath product that
$$
F(E_{ij},O_m) = \sum_{k=0}^{m^n-1}E_{i+kn,j+kn}
$$
so that, with respect to the description given in Lemma \ref{lemmadani0}, the only nonzero blocks of order $n$ are the $m^n$ diagonal blocks of type $E_{ij}$. On the other hand, we have:
$$
F(O_n, Circ_j)  = \sum_{i=1}^n I_m^{\otimes^{i-1}} \otimes Circ_j \otimes I_m^{\otimes^{n-i}}\otimes C_i,
$$
so that in this case we only have blocks of type $C_k$, for $k=1,\ldots, n$. Observe that
$$
F\left(\sum_{i=1}^nE_{ii}, -Circ_0\right)=  F(I_n,-I_m)=  O_{nm^n},
$$
as it follows from the description of $\ker F$ given in Proposition \ref{kernel}. In particular, if we put
$$
F_A = F\left(M_n(\C)\times \{O_m\}\right) \qquad F_B = F\left(\{O_n\}\times Circ_m(\C)\right),
$$
we have $\dim (F_A \cap F_B)=1$, and more specifically $F(\sum_{i=1}^nE_{ii}, O_{m}) = F(O_n, Circ_0) = I_{nm^n}$.

The spectral properties of circulant matrices enable us to give the following result.

\begin{theorem}\label{theoremspectrummain}
Let $A\in \mathcal{M}_n(\mathbb{C})$, and let $B\in Circ_m(\mathbb{C})$  be a circulant matrix as in \eqref{aggiuntaora}. Then the spectrum $\Sigma$ of $A\wr B$ is obtained by taking the union of the spectra $\Sigma_{i_1,\ldots,i_n}$ of the $m^n$ matrices of order $n$ given by
\begin{eqnarray}\label{matrixin}
\widetilde{M}^{i_1,i_2,\ldots, i_n}=A + \sum_{t=1}^n\sum_{i=0}^{m-1}b_i\rho^{ii_t}C_t,
\end{eqnarray}
where $i_j \in \{0,1,\ldots, m-1\}$, for every $j=1,\ldots, n$, and $\rho=\exp\left(\frac{2\pi i}{m}\right)$.
\end{theorem}

\begin{proof}
It follows from the definition of wreath product that, if $B$ is a circulant matrix, then $A\wr B$ can be regarded as a matrix of type
$$
A\wr B = \begin{pmatrix} M_0 & M_1 &  &  & M_{m-1} \\
 M_{m-1} & M_0 & M_1 &  &  \\
 & \ddots & \ddots & \ddots  & \\
  &  & \ddots & \ddots  & M_1 \\
      M_1 &  &  & M_{m-1} & M_0
    \end{pmatrix},
$$
where each block $M_i$ has order $nm^{n-1}$, for every $i=0,\ldots, m-1$. In other words, the matrix $A\wr B$ is a block circulant matrix, with blocks of order $nm^{n-1}$. More precisely, we have:
$$
M_0 = I_m^{\otimes^{n-1}} \otimes A + b_0 I_m^{\otimes^{n-1}}\otimes C_1 + \sum_{i=1}^{n-1}I_m^{\otimes^{i-1}}\otimes B \otimes I_m^{\otimes^{n-i-1}}\otimes C_{i+1}
$$
and
$$
M_i = b_i I_m^{\otimes^{n-1}}\otimes C_1, \qquad \forall i=1,\ldots, m-1.
$$
The spectral analysis developed in \cite{tee} for block circulant matrices implies that the spectrum of $A\wr B$ is given by the union of the spectra of the $m$ matrices
$$
\widetilde{M}^{i_1}=\sum_{h=0}^{m-1}\rho^{hi_1}M_h, \qquad \text{for } i_1=0,\ldots, m-1 \ \text{ and } \ \rho= \exp\left(\frac{2\pi i}{m}\right).
$$
By using the explicit expression of the matrices $M_h$'s given above, we obtain that also the matrix $\widetilde{M}^{i_1}$ is a block circulant matrix, of type
$$
\widetilde{M}^{i_1}=\begin{pmatrix}
      \widetilde{M}^{i_1}_0 & \widetilde{M}^{i_1}_1 &  &  & \widetilde{M}^{i_1}_{m-1} \\
\widetilde{M}^{i_1}_{m-1} & \widetilde{M}^{i_1}_0 & \widetilde{M}^{i_1}_1 &  &  \\
    & \ddots & \ddots & \ddots  & \\
 &  & \ddots & \ddots  & \widetilde{M}^{i_1}_1 \\
 \widetilde{M}^{i_1}_1 &  &  & \widetilde{M}^{i_1}_{m-1} & \widetilde{M}^{i_1}_0
    \end{pmatrix}
$$
where each block has order $nm^{n-2}$. In particular, we have:
\begin{eqnarray*}
\widetilde{M}^{i_1}_0 &=& I_m^{\otimes^{n-2}} \otimes A + b_0 I_m^{\otimes^{n-2}}\otimes C_2 \\
&+& \sum_{i=1}^{n-2} I_m^{\otimes^{i-1}}\otimes B \otimes I_m^{\otimes^{n-i-2}}\otimes C_{i+2}+ \sum_{i=0}^{m-1}b_i\rho^{ii_1}I_m^{\otimes^{n-2}}\otimes C_1
\end{eqnarray*}
and
$$
\widetilde{M}^{i_1}_i = b_iI_m^{\otimes^{n-2}}\otimes C_2, \qquad \forall i=1,\ldots, m-1.
$$
Now we can apply again the theory from \cite{tee}, and we deduce that the spectrum of the matrix $\widetilde{M}^{i_1}$, for each $i_1=0,\ldots, m-1$, is obtained as the union of the spectra of the $m$ matrices
$$
\widetilde{M}^{i_1,i_2}=\sum_{h=0}^{m-1}\rho^{hi_2}\widetilde{M}^{i_1}_h, \qquad \textrm{with } i_2= 0,\ldots, m-1.
$$
Observe that in the expression of $\widetilde{M}^{i_1}_h$ there is one Kronecker product less than in $M_h$.\\ Now, the matrix $\widetilde{M}^{i_1,i_2}$ is still a block circulant matrix with $m\times m$ blocks, each of order $nm^{m-3}$. By iterating this argument, we conclude that the spectrum of the matrix $A\wr B$ is obtained by taking the union of the spectra of the $m^n$ matrices of order $n$
\begin{eqnarray*}
\widetilde{M}^{i_1,i_2,\ldots, i_n}=A + \sum_{t=1}^n\sum_{i=0}^{m-1}b_i\rho^{ii_t}C_t,
\end{eqnarray*}
where the $n$-tuple $(i_1,i_2,\ldots, i_n)$ varies in $\{0,1,\ldots, m-1\}^n$. The claim follows.
\end{proof}

In the case $n=2$, we are able to give an explicit description of the spectrum of $A\wr B$.
\begin{corollary}\label{caseA2}
Let $A=\begin{pmatrix} a & b\\
 c & d
 \end{pmatrix}\in \mathcal{M}_2(\mathbb{C})$ and let $B\in Circ_m(\mathbb{C})$ be a circulant matrix as in \eqref{aggiuntaora}. Then the spectrum $\sum$ of $A\wr B$ is given by
\begin{eqnarray*}
\sum = \bigcup_{i_1,i_2=0,1, \ldots, m-1}\Sigma_{i_1,i_2},
\end{eqnarray*}
with
$$
\Sigma_{i_1,i_2} = \left\{ \frac{\sum_{i=0}^{m-1}b_i(\rho^{ii_1}+\rho^{ii_2})+a+d}{2}\pm\frac{1}{2}\sqrt{\left(\sum_{i=1}^{m-1}b_i(\rho^{ii_1}-\rho^{ii_2})+a-d\right)^2+4bc}   \right\},
$$
and $\rho=\exp\left(\frac{2\pi i}{m}\right)$.
\end{corollary}

\begin{proof}
Let $$
A= \begin{pmatrix} a & b\\
 c & d
 \end{pmatrix} \qquad
B = \begin{pmatrix} b_0 & b_1 &  &  & b_{m-1} \\
 b_{m-1} & b_0 & b_1 &  &  \\
 & \ddots & \ddots & \ddots  & \\
       &  & \ddots & \ddots  & b_1 \\
 b_1 &  &  & b_{m-1} & b_0
 \end{pmatrix}
$$
and let $i_1,i_2\in \{0,\ldots, m-1\}$. Then the matrix \eqref{matrixin} reduces to
$$
\widetilde{M}^{i_1,i_2}= \left(
                           \begin{array}{cc}
                             a+\sum_{i=0}^{m-1}b_i\rho^{ii_1} & b \\
                             c & d+\sum_{i=0}^{m-1}b_i\rho^{ii_2} \\
                           \end{array}
                         \right).
$$
An explicit computation shows that this matrix has spectrum
$$
\Sigma_{i_1,i_2} = \left\{ \frac{\sum_{i=0}^{m-1}b_i(\rho^{ii_1}+\rho^{ii_2})+a+d}{2}\pm\frac{1}{2}\sqrt{\left(\sum_{i=1}^{m-1}b_i(\rho^{ii_1}-\rho^{ii_2})+a-d\right)^2+4bc}   \right\},
$$
and the claim follows.
\end{proof}

\begin{example}\label{exexex}\rm
Let $A= \left(
          \begin{array}{cc}
            1 & 1 \\
            0 & 2 \\
          \end{array}
        \right)$ and $B=\left(
                          \begin{array}{ccc}
                            1 & 2 & -1 \\
                            -1 & 1 & 2 \\
                            2 & -1 & 1 \\
                          \end{array}
                        \right)
$. A direct computation shows that the eigenvalues of the matrix $A\wr B$ are
$$
4;3; \frac{5}{2}\pm \frac{3\sqrt{3}}{2}i; \frac{3}{2}\pm \frac{3\sqrt{3}}{2}i,
$$
where each eigenvalue occurs with multiplicity $3$. The spectrum $\sum_{i_1,i_2}$ of the matrix $\widetilde{M}^{i_1,i_2}$, for all $i_1,i_2=0,1,2$, is described in Table \ref{finanza}.
\begin{table}\small
\begin{tabular}{|c|c|c|}
\hline
$i_1$ & $i_2$ & $\sum_{i_1,i_2}$ \\
\hline
$0$ & $0$ & $3;4$\\
\hline
$0$ & $1$ & $3;  \frac{5}{2}+ \frac{3\sqrt{3}}{2}i$\\
\hline
$0$ & $2$ & $3;  \frac{5}{2}- \frac{3\sqrt{3}}{2}i$\\
\hline
 $1$ & $0$ & $4;  \frac{3}{2}+ \frac{3\sqrt{3}}{2}i$\\
\hline
 $1$ & $1$ &  $\frac{3}{2}+ \frac{3\sqrt{3}}{2}i;  \frac{5}{2}+ \frac{3\sqrt{3}}{2}i$\\
\hline
  $1$ & $2$ &  $\frac{3}{2}+ \frac{3\sqrt{3}}{2}i;  \frac{5}{2}- \frac{3\sqrt{3}}{2}i$\\
\hline
 $2$ & $0$ &  $\frac{3}{2}- \frac{3\sqrt{3}}{2}i;  4$\\
\hline
  $2$ & $1$ &  $\frac{3}{2}- \frac{3\sqrt{3}}{2}i;  \frac{5}{2}+ \frac{3\sqrt{3}}{2}i$\\
\hline
  $2$ & $2$ &  $\frac{3}{2}- \frac{3\sqrt{3}}{2}i;  \frac{5}{2}- \frac{3\sqrt{3}}{2}i$\\
\hline
\end{tabular} \caption{Spectrum of the matrix $A\wr B$ of Example \ref{exexex}.}    \label{finanza}
\end{table}
\end{example}

\subsection{The case of a diagonal $A$ and a uniform $B$.}

Consider now the special case where $A$ is a diagonal matrix and $B$ is a multiple of the uniform matrix $J_m$, so that
$$
A = \left(
      \begin{array}{cccc}
        a_1 &  &  &  \\
         &  a_2  &  \\
         &  & \ddots &  \\
         &  &  & a_n \\
      \end{array}
    \right), \ a_i\in \mathbb{C} \qquad \text{and }  \qquad B = h\left(
                                                  \begin{array}{cccc}
                                                    1 & \cdots & \cdots  & 1 \\
                                             \vdots  & \ddots &  & \vdots \\
                                                    \vdots  &  & \ddots & \vdots \\
                                                    1 & \cdots & \cdots & 1 \\
                                                  \end{array}
                                                \right), \ h\in \mathbb{C}.
$$
The spectrum of $A\wr B$, under these hypotheses, is explicitly described in the following theorem.
\begin{theorem}\label{theoremdeterminant}
Let $A\in \mathcal{M}_n(\mathbb{C})$ be a diagonal matrix, with $k$-th diagonal entry equal to $a_k\in \mathbb{C}$, and let $B = hJ_m$, where $h\in\mathbb{C}$ and $J_m$ is the uniform matrix of order $m$. Then the spectrum of $A\wr B$ is
$$
\bigcup_{(i_1, \ldots, i_n)\in \{0,\ldots,m-1\}^n} \bigcup_{k=1}^{n}\left\{a_k+h\sum_{j=0}^{m-1}\rho^{ji_k} \right\}, \qquad \rho= \exp\left(\frac{2\pi i}{m}\right).
$$
Therefore
$$
\det(A\wr B) = \prod_{k=1}^n (a_k+mh)^{m^{n-1}}\cdot a_k^{(m-1)m^{n-1}}.
$$
\end{theorem}
\begin{proof}
With the notation used in Equation \eqref{matrixin}, we have in this special case
$$
\widetilde{M}^{i_1,i_2,\ldots, i_n}=\left(
      \begin{array}{cccc}
        a_1 &  &  &  \\
         &  a_2  &  \\
         &  & \ddots &  \\
         &  &  & a_n \\
      \end{array}
    \right) + h\sum_{t=1}^n\sum_{j=0}^{m-1}\rho^{ji_t}C_t,
$$
so that the matrix $\widetilde{M}^{i_1,i_2,\ldots, i_n}$ is a diagonal matrix, whose $k$-th diagonal entry is equal to $a_k+h\sum_{j=0}^{m-1}\rho^{ji_k}$, with $\rho = \exp\left(\frac{2\pi i}{m}\right)$, for each $k=1,\ldots,n$.  This implies that the spectrum of the matrix $A\wr B$ is given by
$$
\bigcup_{(i_1, \ldots, i_n)\in \{0,\ldots,m-1\}^n} \bigcup_{k=1}^{n}\left\{a_k+h\sum_{j=0}^{m-1}\rho^{ji_k} \right\}.
$$
Now, observe that
$$
a_k+h\sum_{j=0}^{m-1}\rho^{ji_k} = \left\{
                                     \begin{array}{ll}
                                       a_k+mh, & \hbox{for } i_k=0; \\
                                       a_k, & \hbox{otherwise}
                                     \end{array}
                                   \right.
$$
since $\sum_{j=0}^{m-1}\rho^{ji_k} = \frac{(\rho^{i_k})^m-1}{\rho^{i_k}-1} = 0$, for all $i_k=1,\ldots, m-1$. Therefore, we get:
\begin{eqnarray*}
\det(A\wr B) &=& \prod_{k=1}^n \ \prod_{\substack{(i_1, \ldots, i_n)\in \{0,\ldots,m-1\}^n\\ i_k=0}}(a_k+mh)\cdot \prod_{\substack{(i_1, \ldots, i_n)\in \{0,\ldots,m-1\}^n\\ i_k\neq 0}}a_k\\
&=&\prod_{k=1}^n (a_k+mh)^{m^{n-1}}\cdot a_k^{(m-1)m^{n-1}}.\nonumber
\end{eqnarray*}
\end{proof}

\begin{corollary}\label{otello}
Let $A$ and $B$ be as in Theorem \ref{theoremdeterminant}. Then the matrix $A\wr B$ is singular if and only if there exists an index $k\in \{1,\ldots,n\}$ such that $a_k=0$ or $a_k = -mh$.
\end{corollary}
This result will be applied in Section \ref{sylvestersection} in order to study generalized Sylvester matrix equations by means of the wreath product of matrices.

\subsection{An application to the spectral analysis of the Lamplighter random walk}\label{amilcare}

In this subsection, we investigate the spectrum of the
\lq\lq Walk or switch\rq\rq Lamplighter random walk over a
$d$-regular graph $\mathcal{G}=(V,E)$, already introduced in Section \ref{section3}. For the sake of simplicity, we assume
that the color set consists of just two elements, that can be
identified with the
vertex set of a segment graph (the color graph).

According with Corollary \ref{corollarylamp}, the transition
matrix of this stochastic process is given by
$$
\left(\frac{d}{d+1}A\right)\wr \left(\frac{1}{d+1}B\right),
$$
where $A$ is the normalized adjacency matrix of the graph
$\mathcal{G}$, and $B=\left(
                                            \begin{array}{cc}
                                              0 & 1 \\
                                              1 & 0 \\
                                            \end{array}
                                          \right)$. Observe that, in particular, the $\frac{1}{d+1}B$ is a circulant matrix, so that the spectral
                                           analysis developed in the first part of Section \ref{sectionspectrum} can be applied. In particular, we
                                           know that the spectrum of the matrix $\left(\frac{d}{d+1}A\right)\wr \left(\frac{1}{d+1}B\right)$
                                           is obtained by taking the union of the spectra of the $2^{|V|}$ matrices of order $|V|$ given by:
\begin{eqnarray}\label{dinuovo}
\widetilde{M}^{i_1,i_2,\ldots, i_{|V|}}=\frac{d}{d+1}A + \frac{1}{d+1}
\sum_{t=1}^{|V|}\sum_{i=0}^{1}b_i\rho^{ii_t}C_t,
\end{eqnarray}
where $i_t \in \{0,1\}$, for every $t=1,\ldots, |V|$, and $\rho=-1$.

Let us consider now the specific case where the graph $\mathcal{G}$ is the complete graph $K_n$ over $n$ vertices; observe that this graph is an $(n-1)$-regular graph. With respect to the notation of \eqref{dinuovo}, we have that
$$
A = \frac{1}{n-1}(J_n-I_n) \qquad    B=\left(
                                            \begin{array}{cc}
                                              0 & 1 \\
                                              1 & 0 \\
                                            \end{array}
                                          \right)
$$
where $J_n$ and $I_n$ denote, respectively, the uniform and the
identity matrix of order $n$. Note that we have $b_0=0, b_1=1$, so
that the matrix \eqref{dinuovo} reduces to
\begin{eqnarray*}
\widetilde{M}^{i_1,i_2,\ldots, i_n}=\frac{1}{n}(J_n-I_n) +
\frac{1}{n}\sum_{t=1}^n(-1)^{i_t}C_t =\frac{1}{n} \left(
                                                                                       \begin{array}{ccccc}
                                                                                         (-1)^{i_1} & 1 & 1 & \cdots & 1 \\
                                                                                        1 & (-1)^{i_2} & 1   &  & \vdots \\
                                                                                         1 & 1 & \ddots  &  & \vdots \\
                                                                                     \vdots &  &  & \ddots & 1 \\
                                                                                         1 & \cdots & \cdots  & 1 & (-1)^{i_n} \\
                                                                                       \end{array}
                                                                                     \right).
\end{eqnarray*}

It turns out that the spectrum of the matrix $\widetilde{M}^{i_1,i_2,\ldots, i_n}$ only depends on the number $k$ of indices equal to $1$ in the $n$-tuple $(i_1, i_2, \ldots, i_n)$, whereas it does not depend on the position of such indices. In order to show that, we use the elementary permutation matrix $\mathcal{E}_{ij}$ of order $n$, associated with the transposition $(ij)$ of the symmetric group of order $n$. For instance, for $n=5$, one has $\mathcal{E}_{23}= \left(
                                                                                                                            \begin{array}{ccccc}
                                                                                                                              1 & 0 & 0 & 0 & 0 \\
                                                                                                                              0 & 0 & 1 & 0 & 0 \\
                                                                                                                              0 & 1 & 0 & 0 & 0 \\
                                                                                                                              0 & 0 & 0 & 1 & 0 \\
                                                                                                                              0 & 0 & 0 & 0 & 1 \\
                                                                                                                            \end{array}
                                                                                                                          \right)$. Observe that $\mathcal{E}_{ij}^{-1}=\mathcal{E}_{ij}$, for all $ i,j = 1,\ldots, n$. Let $i_{r_1}, i_{r_2}, \ldots, i_{r_k}$ be the entries of the $n$-tuple $(i_1, i_2, \ldots, i_n)$ which are equal to $1$; similarly, let $j_{s_1}, j_{s_2}, \ldots, j_{s_k}$ be the entries of the $n$-tuple $(j_1, j_2, \ldots, j_n)$ which are equal to $1$. Let $\widetilde{M}^{i_1,i_2,\ldots, i_n}$  and $\widetilde{M}^{j_1,j_2,\ldots, j_n}$ be the corresponding matrices. Then one has
\begin{eqnarray*}
\widetilde{M}^{j_1,j_2,\ldots, j_n} &= & \mathcal{E}_{i_{r_k}j_{s_k}}\cdots \mathcal{E}_{i_{r_1}j_{s_1}}\widetilde{M}^{i_1,i_2,\ldots, i_n}\mathcal{E}_{i_{r_1}j_{s_1}}\cdots \mathcal{E}_{i_{r_k}j_{s_k}} \\
&=&  (\mathcal{E}_{i_{r_1}j_{s_1}}\cdots \mathcal{E}_{i_{r_k}j_{s_k}})^{-1}  \widetilde{M}^{i_1,i_2,\ldots, i_n}\mathcal{E}_{i_{r_1}j_{s_1}}\cdots \mathcal{E}_{i_{r_k}j_{s_k}}.
\end{eqnarray*}
Therefore, the matrices $\widetilde{M}^{i_1,i_2,\ldots, i_n}$  and $\widetilde{M}^{j_1,j_2,\ldots, j_n}$ are conjugate, so that they have the same characteristic polynomial, and so they have the same spectrum.

As a consequence, for each $k=0,1,\ldots, n$, we can reduce to
investigate the spectrum of the matrix $\widetilde{M}^{1,\ldots,
1,0,\ldots, 0}$, corresponding to the $n$-tuple starting with $k$
entries equal to $1$, whose remaining $n-k$ entries are equal to
$0$. Put $\widetilde{M}^{1,\ldots, 1,0,\ldots, 0}=\frac{1}{n}P$,
with
$$
P = J_n + Q,
$$
where $Q$ is the diagonal matrix $Q = \left(
                                        \begin{array}{cccccc}
                                          -2 &  &  &  &  &  \\
                                           & \ddots &  &  &  &  \\
                                           &  & -2&  &  &  \\
                                           &  &  & 0&  &  \\
                                           &  & &  & \ddots &  \\
                                           &  &  &  &  & 0 \\
                                        \end{array}
                                      \right)
$. Now we have:
\begin{eqnarray*}
\det (\lambda I_n - P) & = & \det \left( \lambda I_n -J_n-Q \right) \\
&=& \det \left( (\lambda I_n - Q) \left(I_n-(\lambda I_n-Q)^{-1}J_n \right)\right)\\
&=& \det(\lambda I_n -Q)\cdot \det \left( I_n-(\lambda
I_n-Q)^{-1}J_n \right).
\end{eqnarray*}
It is clear that
$$
\det(\lambda I_n -Q) = (\lambda+2)^k\cdot \lambda^{n-k}.
$$
Now it can be seen that the matrix $(\lambda I_n-Q)^{-1}J_n$ is
the matrix of rank $1$, whose first $k$ rows are
constant, with entries all equal to $\frac{1}{\lambda+2}$, whereas
the remaining $n-k$ rows are constant, and with entries all equal
to $\frac{1}{\lambda}$. Therefore, $(\lambda I_n-Q)^{-1}J_n$ has $n-1$ eigenvalues equal to $0$, and one eigenvalue equal to $\frac{k}{\lambda+2}+\frac{n-k}{\lambda}$. This implies that the matrix $ I_n-(\lambda I_n-Q)^{-1}J_n$ has $n-1$ eigenvalues equal to $1$, and one eigenvalue equal to $1-\frac{k}{\lambda+2}+\frac{n-k}{\lambda}$, so that:
$$
\det \left( I_n-(\lambda I_n-Q)^{-1}J_n \right) = 1 - \frac{k}{\lambda+2} -\frac{n-k}{\lambda}.
$$
Then, by gluing together all the previous computations, we obtain:
\begin{eqnarray}\label{stark}
\det(\lambda I_n -P) = (\lambda+2)^{k-1}\cdot \lambda^{n-k-1}\cdot
(\lambda^2+(2-n)\lambda+2k-2n).
\end{eqnarray}
For the particular value $k=0$, we get:
\begin{eqnarray}\label{star0}
\det(\lambda I_n -P) = \lambda^{n-1}\cdot (\lambda-n);
\end{eqnarray}
similarly, for the special value $k=n$, we have:
\begin{eqnarray}\label{starn}
\det(\lambda I_n -P) = (\lambda+2)^{n-1}\cdot (\lambda+2-n).
\end{eqnarray}
We deduce the following theorem.
\begin{theorem}\label{spectrumlamp}
The spectrum of the transition matrix of the \lq\lq Walk
or switch\rq\rq Lamplighter random walk over the complete graph
$K_n$ is given by $\Sigma = \bigcup_{k=0}^n \binom{n}{k}\Sigma_k$,
where
$$
\Sigma_0 = \{0\ \textrm{ (multiplicity $n-1$)};\ 1\} \qquad
\Sigma_n = \left\{-\frac{2}{n}\ \textrm{ (multiplicity $n-1$)};\
\frac{n-2}{n}\right\}
$$
and, for each $k=1,\ldots, n-1$:
$$
\Sigma_k = \left\{-\frac{2}{n}\ \textrm{ (multiplicity $k-1$)}; \
0\ \textrm{ (multiplicity $n-k-1$)}; \
\frac{n-2\pm\sqrt{(n+2)^2-8k}}{2n} \right\}.$$
\end{theorem}
\begin{proof}
Observe that $\Sigma$ is the spectrum of the matrix $\widetilde{M}^{1,\ldots, 1,0,\ldots, 0}=\frac{1}{n}P$. Therefore, it suffices to normalize by $n$ the roots of the characteristic polynomial \eqref{stark}, for $1\leq k<n$, taken with multiplicity $\binom{n}{k}$; of the characteristic polynomial \eqref{star0} for $k=0$; and of the characteristic polynomial \eqref{starn} for $k=n$.
\end{proof}

\begin{remark}\rm
Notice that in the previous computations we had $B=Circ_1\in \mathcal{M}_{2}(\C)$. In principle (with an amount of computational difficulties) analogous computations can be performed for any circulant matrix $B= \sum_{i=1}^{m-1}t_iCirc_i\in \mathcal{M}_m(\C)$, with $t_i\in \{0,1\}$, provided that such a linear combination gives a symmetric matrix, which can be regarded as the adjacency matrix of some (circulant) color graph. In this case, the total degree of the corresponding wreath product of graphs would be $n-1+|\{ i \textrm{ s.t. } t_i =1\}|$.
\end{remark}

\section{Application to generalized Sylvester matrix equations}\label{sylvestersection}

In this section, we address the question whether it is possible to find conditions to guarantee that the matrix $A\wr B$ has full rank. This problem is related to the solution of the so-called Sylvester matrix equations. J. J. Sylvester studied the equation $AX+XB=C$, where $A\in  \mathcal{M}_n(\mathbb{C})$, $B\in  \mathcal{M}_m(\mathbb{C})$, $C\in  \mathcal{M}_{n\times m}(\mathbb{C})$ are given, and one wants to determine the matrix $X\in \mathcal{M}_{n\times m}(\mathbb{C})$. It is well known that the solution of this equation exists and is unique if and only if $A$ and $-B$ have no common eigenvalue \cite{sylv}.
Actually, Sylvester studied also the more general equation
\begin{equation}\label{sylvester}
\sum_{i=1}^k A_iXB_i=C,
\end{equation}
with $A_i\in  \mathcal{M}_{n\times m}(\mathbb{C})$, $B_i\in  \mathcal{M}_{s\times t}(\mathbb{C})$, and $C\in \mathcal{M}_{n\times t}(\mathbb{C})$. Many methods were developed to solve this kind of equations (see, for instance \cite{ding}).
Here, we are interested in finding a sufficient condition to guarantee that the Equation \eqref{sylvester} admits a unique solution. The connection between the solution of the Equation \eqref{sylvester} and the wreath product of matrices is given by the following lemma. Given an $m\times n$ matrix $M$, we denote by $\VVec(M)$ the column vector of length $mn$ obtained by stacking the columns of $M$ on top of one another.

\begin{lemma}\label{lemma_sylvester}
Let $A_i\in  \mathcal{M}_{n\times m}(\mathbb{C})$, $B_i\in  \mathcal{M}_{s\times t}(\mathbb{C})$, and $C\in  \mathcal{M}_{n\times t}(\mathbb{C})$. The equation $\sum_{i=1}^k A_iXB_i=C$ can be rewritten in the form
$$
(B_1^T \otimes A_1+\cdots +B_k^T \otimes A_k){\rm Vec}(X)={\rm Vec}(C).
$$
\end{lemma}
\begin{proof}
Let us start by considering the case $k=1$, so that we look at the equation $A_1XB_1=C$. We denote by $\{a_{ij}\}_{i=1,\ldots,n; j=1,\ldots,m}$, by $\{b_{ij}\}_{i=1,\ldots,s; j=1,\ldots,t}$, by $\{c_{ij}\}_{i=1,\ldots,n; j=1,\ldots,t}$ and by $\{x_{ij}\}_{i=1,\ldots,m; j=1,\ldots,s}$ the entries of the matrices $A_1, B_1, C, X$, respectively.

Now the entry $c_{ij}$ of $C$ equals $\sum_{h=1}^sy_{ih}b_{hj}$, where $y_{ih}$ is the entry $(i,h)$ of $A_1X$. Therefore:
\begin{eqnarray}\label{lillo}
c_{ij} &=& \sum_{h=1}^s\left(\sum_{\ell=1}^m a_{i\ell}x_{\ell h}\right)b_{hj} \nonumber\\
&=& \sum_{\ell=1}^m\left(\sum_{h=1}^s \widetilde{b_{jh}}a_{i\ell}\right)x_{\ell h},
\end{eqnarray}
where $\widetilde{b_{jh}}$ denotes the entry $(j,h)$ of $B_1^T$. Notice also that $c_{ij}$ occupies the $((j-1)n+i)$-th entry of the column vector $\VVec(C)$, by definition.
On the other hand, \eqref{lillo} is exactly the entry $(j-1)n+i$ of the column vector $(B_1^T\otimes A_1)\VVec (X)$, as it can be regarded as the sum of the componentwise products of the elements of the $((j-1)n+i)$-th row of $B_1^T\otimes A_1$, with the column vector $\VVec(X)$. This concludes the proof for the case $k=1$.

Now for every $k>1$ consider the equation $\sum_{i=1}^k A_iXB_i=C$. By proceeding as before, we are led to the equation
\begin{eqnarray}\label{cappacappa}
c_{ij} = \sum_{\ell=1}^m\left(\sum_{h=1}^s \left(\widetilde{b_{jh}}a_{i\ell}+ \widetilde{b_{jh}'}a_{i\ell}'+ \cdots +  \widetilde{b_{jh}^{(k-1)}}a_{i\ell}^{(k-1)}\right) \right)x_{\ell h},
\end{eqnarray}
where $a_{i\ell}^{(p)}$ denotes the entry $(i,\ell)$ of $A_{p+1}$ and $\widetilde{b_{jh}^{(p)}}$ the entry $(j,h)$ of $B_{p+1}^T$, for each $p=0,\ldots, k-1$. On the other hand, the right-hand side of \eqref{cappacappa} is exactly the entry $(j-1)n+i$ of the column vector $(B_1^T\otimes A_1+\cdots +B_k^T\otimes A_k)\VVec (X)$, as it can be regarded as the sum of the componentwise products of the elements of the $((j-1)n+i)$-th row of $B_1^T\otimes A_1+\cdots +B_k^T\otimes A_k$, with the column vector $\VVec(X)$. This concludes the proof.
\end{proof}

The previous result leads to the following proposition.
\begin{prop}\label{prop_sistema}
Let $A\in \mathcal{M}_n(\C)$, $B\in \mathcal{M}_m(\C)$, and let $C_i$ be as in Definition \ref{maindefinition}, for each $i=1,\ldots, n$. Put
$$
A_1=A, \quad A_2=C_1,\quad \ldots\ldots,\quad A_{n+1}=C_n
$$
and
$$
B_1=I_m^{\otimes^n}, \quad B_2= \left( B \otimes I_m^{\otimes^{n-1}}\right)^T, \ldots, \quad B_h= \left( I_m^{\otimes^{h-2}}\otimes B \otimes I_m^{\otimes^{n-h+1}}\right)^T, \ldots
$$
$$
\ldots, \quad B_{n+1}= \left( I_m^{\otimes^{n-1}}\otimes B\right)^T.
$$
Then the equation $\sum_{i=1}^{n+1} A_iXB_i=C$ has a unique solution if and only if $A\wr B$ has full rank.
\end{prop}
\begin{proof}
The statement follows from Lemma \ref{lemma_sylvester} by observing that the matrix multiplying $\VVec(X)$ coincides with $A\wr B$. Then the linear system $(A\wr B)\VVec(X)=\VVec(C)$ has a unique solution if and only if the matrix $A\wr B$ has full rank.
\end{proof}

It follows that, if $\det(A\wr B)\neq 0$, then the solution of the corresponding matrix equation given in Proposition \ref{prop_sistema} is unique. In particular, the following corollary gives a sufficient condition for the uniqueness of the solution of a special type of Sylvester matrix equations.

\begin{corollary}\label{uniqueness}
Let $A$ and $B$ be as in Theorem \ref{theoremdeterminant}, and suppose  that there exists no index $k\in \{1,\ldots,n\}$ such that $a_k=0$ or $a_k = -mh$. Let $A_1,\ldots, A_{n+1}$ and $B_1,\ldots, B_{n+1}$ be defined as in Proposition \ref{prop_sistema}. Then the equation
$$
 \sum_{i=1}^{n+1} A_iXB_i=C
$$
admits a unique solution $X$.
\end{corollary}
\begin{proof}
It directly follows from Corollary \ref{otello}.
\end{proof}

\section*{Acknowledgments}
We want to express our deepest gratitude to Fabio Scarabotti for enlightening discussions and for his continuous encouragement.


\begin{thebibliography}{99}
\bibitem{bar} L. Bartholdi, W. Woess, Spectral computations on lamplighter groups and Diestel-Leader graphs, {\it
J. Fourier Anal. Appl.} {\bf 11} (2005), no. 2, 175--202. \texttt{doi:10.1007/s00041-005-3079-0}

\bibitem{sylv} G. Birkhoff, S. Mac Lane, A survey of modern algebra, 3rd Edition. {\it The Macmillan Co., New York; Collier-Macmillan Ltd., London}. 1965. x+ 437 pp.

\bibitem{aap} D. D'Angeli, A. Donno, Crested products of Markov chains, {\it Ann. Appl. Probab.} {\bf 19} (2009), no. 1, 414--453. \texttt{doi:10.1214/08-AAP546}

\bibitem{orthogonal} D. D'Angeli, A. Donno, Markov chains on orthogonal block structures, {\it European J. Combin.} {\bf 31} (2010), no. 1, 34--46. \texttt{doi:10.1016/j.ejc.2009.02.003}

\bibitem{generalized} D. D'Angeli, A. Donno, Generalized crested products of Markov chains, {\it European J. Combin.} {\bf 32} (2011), no. 2, 243--257.     \texttt{doi:10.1016/j.ejc.2010.09.007}

\bibitem{aam} D. D'Angeli, A. Donno, The lumpability property for a family of Markov chains on poset block structures, {\it Adv. in Appl. Math.} {\bf 51} (2013), no. 3, 367--391. \texttt{doi:10.1016/j.aam.2013.04.007}

\bibitem{davis} P. J. Davis, Circulant matrices. A Wiley-Interscience
Publication. {\it Pure and Applied Mathematics}. John Wiley \&
Sons, New York-Chichester-Brisbane, 1979. xv + 250 pp.

\bibitem{ding} F. Ding, P. X. Liu, J. Ding, Iterative solutions of the generalized Sylvester matrix equations by using the hierarchical
identification principle, {\it Appl. Math. Comput.} {\bf 197} (2008), no.1, 41--50.  \texttt{doi:10.1016/j.amc.2007.07.040}

\bibitem{ijgt} A. Donno, Replacement and zig-zag products, Cayley graphs and Lamplighter random walk, {\it Int. J. Group Theory} {\bf 2} (2013), no. 1, 11--35.

\bibitem{gc} A. Donno, Generalized wreath products of graphs and groups, {\it Graphs Combin.} {\bf 31} (2015), no. 4, 915--926. \texttt{doi:10.1007/s00373-014-1414-4}

\bibitem{erschler} A. Erschler, Generalized wreath products, {\it Int. Math. Res. Notices} 2006, Art. ID 57835, 14 pp. \texttt{doi:10.1155/IMRN/2006/57835}

\bibitem{imrichbook} R. Hammack, W. Imrich, S. Klav\v{z}ar, Handbook of product graphs. Second edition.
{\it Discrete Mathematics and its Applications (Boca Raton)}. CRC
Press, Boca Raton, FL (2011).

\bibitem{imrich} W. Imrich, H. Izbicki, Associative Products of
Graphs, {\it Monatsh. Math.} {\bf 80} (1975), no. 4, 277--281.  \texttt{doi:10.1007/BF01472575}

\bibitem{sabidussi} G. Sabidussi, The composition of graphs, {\it Duke Math.
J.} {\bf 26}, 693--696 (1959).

\bibitem{sabidussi2} G. Sabidussi, Graph multiplication, {\it Math.
Z.} {\bf 72} (1959/1960), 446--457.  \texttt{doi:10.1007/BF01162967}

\bibitem{sca1} F. Scarabotti, F. Tolli, Harmonic
Analysis of finite lamplighter random walks, {\it J. Dyn. Control
Syst.} {\bf 14} (2008), no. 2, 251--282.   \texttt{doi:10.1007/s10883-008-9038-8}

\bibitem{sca2} F. Scarabotti, F. Tolli, Harmonic analysis on a finite
homogeneous space, {\it Proc. Lond. Math. Soc. (3)} {\bf 100}
(2010), no. 2, 348--376.  \texttt{doi:10.1112/plms/pdp027}

\bibitem{det} J. R. Silvester, Determinants of block matrices, {\it Math. Gazette} {\bf 84} (2000), no. 501, 460--467.

\bibitem{tee} G. J. Tee, Eigenvectors of block circulant and alternating circulant matrices, {\it New Zealand J. Math.} {\bf 36} (2007), 195--211.

\bibitem{woe} W. Woess, A note on the norms of transition
operators on lamplighter graphs and groups, {\it Internat. J.
Algebra Comput.} {\bf 15} (2005), no. 5--6, 1261--1272.  \texttt{doi:10.1142/S0218196705002591}
\end{thebibliography}
\end{document}